\documentclass[11pt, a4paper]{amsart}
\usepackage{amsthm, amssymb, mathtools, cite, graphicx, color,, subcaption, amsmath}
\usepackage{tikz}
\usepackage{tikz-cd}

\usetikzlibrary{matrix}
\mathtoolsset{showonlyrefs}
\newtheorem{theorem}{Theorem}[section]
\newtheorem*{theorem*}{Theorem}

\newtheorem{proposition}[theorem]{Proposition}
\newtheorem{corollary}[theorem]{Corollary}
\newtheorem*{corollary*}{Corollary}
\newtheorem{definition}[theorem]{Definition}
\newtheorem{conjecture}[theorem]{Conjecture}
\numberwithin{equation}{section}
\title{Hitchin Harmonic Maps are Immersions}
\author{Andrew Sanders}

\begin{document}

\address{Department of Mathematics, Statistics and Computer Science, University of Illinois at Chicago, Chicago, IL 60607 USA}
\email{andysan@uic.edu}
\thanks{Sanders gratefully acknowledges partial support from the National Science Foundation Postdoctoral Research Fellowship 1304006 and from U.S. National Science Foundation grants DMS 1107452, 1107263, 1107367 "RNMS: GEometric structures And Representation varieties" (the GEAR Network).}


\date{}


\begin{abstract}
In \cite{HIT92}, Hitchin used his theory of Higgs bundles to construct an important family of representations $\rho:\pi_1(\Sigma)\rightarrow G^{r}$ where $\Sigma$ is a closed, oriented surface of genus at least two, and $G^{r}$ is the split real form of a complex adjoint simple Lie group $G.$  These \textit{Hitchin} representations comprise a component of the space of conjugacy classes of representations of $\pi_1(\Sigma)$ into $G^{r}$ and are deformations of the irreducible Fuchsian representations $\pi_1(\Sigma)\rightarrow \text{SL}_2{\mathbb{R}}\rightarrow G^{r}$
which uniformize the surface $\Sigma.$  For any choice of marked complex structure on the surface and any Hitchin representation, we show that the corresponding equivariant harmonic map is an immersion.  Pulling back the Riemannian metric on the symmetric space $G/K,$ we construct a map from the space of Hitchin representations to the space of isotopy classes of Riemannian metrics on the surface $\Sigma.$  As an application of this procedure, we obtain a new lower bound on the exponential growth rate of orbits in $G/K$ under the action of the image of a Hitchin representation.
\end{abstract}
\maketitle

\section{Introduction}

Equivariant harmonic maps became an important tool to study properties of representations of fundamental groups of closed manifolds some time after the foundations of the theory were laid by Eells-Sampson \cite{ES64}.  Many of the more remarkable applications were to rigidity theorems, of which the work of Toledo \cite{TOL89}, Gromov-Schoen \cite{GS92}, and Corlette \cite{COR95} are just a few important examples, with a primary tool being the Bochner formula for harmonic maps.  At around the same time, the theory of Higgs bundles was introduced by Hitchin \cite{HIT87} and further developed by Simpson \cite{SIM92}, with important analytic foundations laid by Donaldson \cite{DON87} and Corlette \cite{COR88}, showing that equivariant harmonic maps could be used as a bridge between representation varieties for fundamental groups of compact, K\"{a}hler manifolds (of which Riemann surfaces are the most well developed, and rich, examples) and the complex analytic theory of Higgs bundles.  This dictionary came to be known as the \textit{non-abelian Hodge correspondence}, and while the theory provides an amazing connection between seemingly disparate subjects, perhaps as a result of this unexpected relationship, it is very difficult to precisely translate information from one side of the story to the other.  The aim of the present paper is to show that some information about the basic properties of equivariant harmonic maps \textit{can} be rather easily read off from the data of Higgs bundles.

A harmonic map $f:X\rightarrow M$ from a Riemann surface to a Riemannian manifold $M$ is an extremizer of the (conformally invariant) Dirichlet energy,
\begin{align}
E(f)=\frac{1}{2}\int_{X}\lVert df\rVert^2 dV.
\end{align}
Just as the theory of harmonic functions in \textit{real} dimension two is inextricably linked to the theory of holomorphic functions in \textit{complex} dimension one, harmonic maps $f:X\rightarrow M$ also give rise to holomorphic objects.  Precisely, the $(1,0)$-part of the differential $df: TX\rightarrow f^{*}(TM)$ is a holomorphic section of the canonical bundle $K$ of $X$ twisted by the pull-back of the complexified tangent bundle $f^{*}(TM)\otimes \mathbb{C}$ of $M$ with a suitable complex structure arising from the Levi-Civita connection on $M.$  The theory of Higgs bundles, seen in this guise, gives an integrability condition: given a holomorphic section $\phi$ of $K\otimes (TM\otimes \mathbb{C}),$ it describes when this holomorphic section \textit{is} the $(1,0)$-part of the differential of a harmonic map $f:X\rightarrow M.$  Thus, this allows us to determine properties of the differential of harmonic maps (such as the rank), by analyzing this associated holomorphic section.  We aim to apply this process in the case of harmonic maps which are equivariant for \textit{Hitchin} representations.

In \cite{HIT92}, Hitchin discovered a special component of the space of conjugacy classes of representations $\pi_1(\Sigma)\rightarrow G^{r}$ where $\Sigma$ is a closed, oriented surface of genus at least two, and $G^r$ is the split real form of a complex adjoint simple Lie group $G.$  As a fundamental example consider $G=\text{PSL}_n(\mathbb{C})$ and $G^{r}=\text{PSL}_n(\mathbb{R}).$  This discovery was made via choosing a complex structure (a polarization) $\sigma$ on $\Sigma$ and identifying a component of the space of poly-stable $G$-Higgs bundles on $(\Sigma,\sigma),$ which via the non-abelian Hodge correspondence Hitchin \cite{HIT87} and Simpson \cite{SIM92} had developed earlier, yielded representations $\pi_1(\Sigma)\rightarrow G.$  Hitchin used gauge theoretic techniques to conclude that this component of representations consisted of $\rho:\pi_1(\Sigma)\rightarrow G^{r}$ taking values in the split real form of $G.$  Moreover, this component was characterized by the fact that it contained all of the Fuchsian uniformizing representations $\rho:\pi_1(\Sigma)\rightarrow \text{SL}_{2}(\mathbb{R})\rightarrow G^{r}$ where the latter arrow indicates the unique, irreducible representation arising from exponentiating the principal $3$-dimensional sub-algebra inside the Lie algebra of $G.$  It was originally discovered by Goldman \cite{GOL90} for $G^{r}=\text{SL}_3(\mathbb{R})$ (before the paper of Hitchin), and then later, by Labourie in general, that every such \textit{Hitchin} representation was discrete and faithful, with every element $\rho(\gamma)$ diagonalizable with distinct real eigenvalues (purely loxodromic).  

Most of the progress in understanding the geometry of Hitchin representations has since arisen from dynamical and topological methods, most notably through the important introduction by Labourie of the notion of an \textit{Anosov} representation \cite{LAB06}.  The purpose of this paper is to indicate that the theory of Higgs bundles, which was the original genesis of Hitchin representations, can be used to reveal some of the geometry in this context.  In particular, given any Hitchin representation $\rho:\pi_1(\Sigma)\rightarrow G^{r},$ a foundational result of Corlette \cite{COR88} implies the unique existence of an equivariant harmonic map,
\begin{align}
f:(\widetilde{\Sigma},\sigma)\rightarrow G/K,
\end{align}
where $K\subset G$ is the maximal compact subgroup.  Using the Higgs bundle parameterization our main theorem is the following basic property of all these harmonic maps:
\begin{theorem}\label{mainthm}
Let $G$ be a complex adjoint simple Lie group and $\rho\in \text{Hit}(G)$ a Hitchin representation.  For any marked complex structure $\sigma$ on $\Sigma,$ let $f_{\rho}:(\widetilde{\Sigma},\sigma)\rightarrow G/K$ be the corresponding $\rho$-equivariant harmonic map.  Then $f_{\rho}$ is an immersion.
\end{theorem}
In the case $G=\text{PSL}_2(\mathbb{C}),$ this was observed by Hitchin \cite{HIT87} in his original paper introducing Higgs bundles through an application of a general theorem of Schoen-Yau \cite{SY78} about harmonic maps between Riemann surfaces which are homotopic to diffeomorphisms.

There is an interesting, and not thoroughly explored, relationship between the immersiveness of equivariant harmonic maps $\widetilde{\Sigma}\rightarrow G/K$ and the algebra/geometry of the corresponding representations $\pi_1(\Sigma)\rightarrow G.$  When $G=\text{PSL}_2(\mathbb{R}),$ one can use the well-developed harmonic map theory to assert that an equivariant harmonic map $\widetilde{\Sigma}\rightarrow \mathbb{H}^2$ is an immersion if and only if the associated representation if Fuchsian.  A theorem of Goldman \cite{GOL80} verifies that this is equivalent to the representation being discrete and faithful.  Though we do not analyze the situation in this paper, this line of ideas can not be carried over without change in the higher rank case.  In particular, when $G^{r}=\text{PSp}_4(\mathbb{R}),$ Bradlow, Garcia-Prada, and Gothen \cite{BGG12} have discovered many components of representations $\pi_1(\Sigma)\rightarrow \text{PSp}_4(\mathbb{R})$ which are not Hitchin, but are nonetheless entirely composed of discrete, faithful representations.  It seems very likely that the techniques of the current paper can be used to show that the corresponding equivariant harmonic maps in those cases are also immersions.  It is unclear in what generality the relationship between discrete and faithful representations and equivariant harmonic maps being immersions can be pushed, mostly due to the highly transcendental nature of constructing harmonic maps.  We state the following as a question:

\textbf{Question:}  Let $\rho:\pi_1(\Sigma)\rightarrow G$ be a discrete, faithful, reductive representation where $G$ is a complex, semi-simple Lie group.  Under what conditions are the corresponding equivariant harmonic maps immersions?

The penultimate section of the paper puts forward a dynamical application of Theorem \ref{mainthm}.  The primary result concerns the study of the exponential growth rate of orbits in the symmetric space $G/K$ for the group $\Gamma=\rho(\pi_1(\Sigma))$ acting via isometries, where $\rho$ is a Hitchin representation.  We produce a lower bound on this exponential growth rate in terms of the data of the average sectional curvature of an immersed, equivariant minimal surface.
\begin{theorem}
Let $\rho:\pi_{1}(\Sigma)\rightarrow G$ be a Hitchin representation where $G$ is a complex adjoint simple Lie group $G.$  Then there exists
\begin{align}
f_{\rho}:\widetilde{\Sigma}\rightarrow G/K,
\end{align}  
a $\rho$-equivariant immersed minimal surface.  Let $(g,B)$ denote the induced metric and second fundamental form of this immersion.
Let $h(\rho)$ be the volume entropy of $\rho.$  Then,
\begin{align}
\frac{1}{\text{\textnormal{Vol}}(g)}\int_{\Sigma}\sqrt{-\text{\textnormal{Sec}}(T_{f_{\rho}}(f_{\rho}(\widetilde{\Sigma})))+\frac{1}{2}\lVert B\rVert_{g}^2} \ dV_{g} \leq h(\rho).
\end{align}
The first term under the radicand is the ambient sectional curvature of the tangent two plane to the minimal surface in the symmetric space $G/K.$
\end{theorem}
This theorem is an immediate extension of work the author did in the setting of negative curvature in \cite{SAN14}.  It is interesting to note that in that case, the principal application was to the study of representations which were very close to being Fuchsian, while here the main source of information is gained from representations which are very far from being Fuchsian.

The above inequality gives a formally cheap proof that this exponential growth rate is always positive, which has been shown using much more sophisticated techniques by Sambarino \cite{SAM12} (Sambarino shows much more than the positivity, he shows that this growth rate is \textit{exactly} exponential).  Perhaps the most interesting consequence of this is the following: if the exponential growth rate goes to zero along a sequence of Hitchin representations, then the corresponding equivariant minimal surfaces in $G/K$ become increasingly flat.  Recent work by Katzarkov-Noll-Pandit-Simpson \cite{KNPS13} aims to identify a limiting harmonic map in this situation from $\widetilde{\Sigma}$ to a "universal building."  The notion that the approximating equivariant minimal surfaces are getting more and more flat seems congruent with the picture developed in \cite{KNPS13}.  Additionally, recent work of Collier-Li \cite{CL14} has also exposed this asymptotic flatness phenomenon.

As a last application, for each choice of marked complex structure $\sigma$ on $\Sigma,$ we define a map
\begin{align}
\mathfrak{U}_{\sigma}:\text{Hit}(G)\rightarrow \mathcal{M}(\Sigma),
\end{align}
 from the space of Hitchin representations $\text{Hit}(G)$ to the deformation space of Riemannian metrics $\mathcal{M}(\Sigma)$ on the surface $\Sigma,$ simply by pulling back the Riemannian metric of the symmetric space through the corresponding equivariant harmonic map.  By construction this map is equivariant under the mapping class group action.  The fact that the harmonic map is always an immersion ensures that the these pull-back tensors \text{are} Riemannian metrics.  Unfortunately, this map has deficits.  Most importantly, the map is \textit{not} injective, the circle action defined in the realm of Higgs bundles often produces harmonic maps whose pull-back metrics are isometric provided the rank of $G$ is greater than one.  We make one final comment that the $L^2$-metric on the deformation space of Riemannian metrics $\mathcal{M}(\Sigma)$ equips this pictures with some Riemannian geometry.  In addition, it is well-known that the restriction of the $L^2$-metric to the Fuchsian representations coincides with the Weil-Petersson metric under the identification of the space of Fuchsian representations with the Teichm\"{u}ller space of isotopy classes of complex structures on $\Sigma.$  We hope to study this geometry in a future paper.

\section{Acknowledgements}

This paper hugely benefitted from conversations with many mathematicians over the course of many years.  I am particularly grateful to Richard Wentworth and William Goldman for spending years teaching me about the theory of Higgs bundles.  Additionally, this paper would not have come to fruition without numerous conversations with Brian Collier, who helped me tremendously in understanding much of the Lie theory relevant to Hitchin representations.

\section{Background on Harmonic maps}\label{harmmaps}
Throughout, $\Sigma$ is a smooth, connected, oriented surface.  For the general (local) discussion below, we do not require $\Sigma$ to be closed, although later in the paper this will be the case.

This exposition draws on heavily on the discussion found in \cite{MOO06}.  Let $X=(\Sigma,\sigma)$ be a Riemann surface (here $\sigma$ is a complex structure) and $(M,h)$ a Riemannian manifold.  Given any $C^{1}$-mapping $f:X\rightarrow M,$ the Dirichlet energy is defined via,
\begin{align}
E(f)=\frac{1}{2} \int_{\Sigma} \lVert df \rVert^2 \ dV_{g},
\end{align}
where the volume element is computed using any conformal metric $g$ on $X$ and the Hilbert-Schmidt norm $\lVert df \rVert^2$ is given in local coordinates by the expression,
\begin{align}
\lVert df \rVert^2=g^{ij}\partial_{i}f^{\alpha}\partial_{j}f^{\beta}h_{\alpha \beta}(f).
\end{align}
It is straightforward to check that the integrand $\lVert df \rVert^2 \ dV_{g}$ is invariant under conformal changes of the metric $g\rightarrow e^{2u}g$ for any $u\in C^{\infty}(\Sigma).$  
\begin{definition}
A $C^{1}$-mapping $f:X\rightarrow M$ is harmonic if it is a critical point of the Dirichlet energy with respect to all compactly supported variations of $f.$
\end{definition}

In this paper, all harmonic maps we meet will be smooth and we shall assume this henceforth.

Note that harmonicity is a well defined notion for maps from the Riemann surface $X$ since the Dirichlet energy is conformally invariant.  For our purposes, it will be important to understand the holomorphic implications of harmonicity.  Suppose $f:X\rightarrow M$ is a harmonic map and consider the differential,
\begin{align}
df: TX \rightarrow f^{*}(TM).
\end{align}
Complexifying the differential and projecting onto the $(1,0)$-part using the complex structure of $X$ defines a map,
\begin{align}
\partial f: K^{-1}\rightarrow f^{*}(TM)\otimes \mathbb{C},
\end{align}
where $K$ is the canonical bundle of holomorphic $1$-forms on $X.$

The Levi-Civita connection of $h$ defines a connection on the pullback bundle $f^{*}(TM).$  The complex linear extension of this connection equips $E=f^{*}(TM)\otimes \mathbb{C}$ with a holomorphic structure by restricting to the $(0,1)$-part (here we use the fact that Cauchy-Riemann operators on complex vector bundles over Riemann surfaces are trivially integrable due to the lack of $(0,2)$ forms on $X).$  This is commonly called the Koszul-Malgrange holomorphic structure.  Globally, we may consider $\partial f$ as a smooth section of the holomorphic vector bundle,
\begin{align}
\text{End}(K^{-1}, E)\simeq K\otimes E,
\end{align}
where $K^{-1}$ has the natural complex structure of the holomorphic tangent bundle to $X.$  The following proposition is essential (see \cite{HW08}).
\begin{proposition}
If a map $f:X\rightarrow M$ is harmonic, then $\partial f$ is a holomorphic section of $\text{\textnormal{End}}(K^{-1},E).$  
\end{proposition}
We remark that the equation which encodes the holomorphicity of $\partial f$ is just a particular form of the Euler-Lagrange equations for the Dirichlet energy.  We refer to reader to the articles \cite{ES64}, \cite{HW08} for a detailed account of the variational description of harmonic maps.

Following standard notation, we shall write $\partial f\in \text{\textnormal{H}}^{0}(\text{\textnormal{End}}(K^{-1}, E)).$
Since $\partial f: K^{-1}\rightarrow E$ is a holomorphic map and $K^{-1}$ is a line bundle over a Riemann surface, basic Riemann surface theory implies that either $\partial f$ is identically zero, or there exists a discrete collection of points $\{ p_i\}$ (the \textit{branch} points) where $\partial f$ vanishes.  For the rest of the discussion we will assume that $X$ is compact.  Choosing a local isothermal coordinate $z_i$ satisfying $z_i(p_i)=0,$ there exists $n_i$ positive integers such that
\begin{align}\label{branch}
\frac{\partial{f}}{\partial z_i}=z_i^{n_i} g_i(z_i),
\end{align}
where $g_i$ is a non-vanishing holomophic function on the coordinate chart.  Away from the $p_i,$ the derivative
\begin{align}
\frac{\partial f}{\partial z},
\end{align}
defines a holomorphic section of the projective bundle
\begin{align}
\mathbb{P}(E).
\end{align}
This defines a complex line in $E$ over the complement of the $\{p_i\}.$  Furthermore, the expression \eqref{branch} implies that we may extend this choice of complex lines over the branch points and obtain a holomorphic sub-bundle $L\subset E$ which by construction is isomporphic to
\begin{align}
L\simeq K^{-1}\otimes \mathcal{O}(D),
\end{align}
where $D$ is an effective divisor given by the expression $D=\sum n_i[p_i].$  In terms of the map $f,$ the branch points $\{p_i\}$ are the points where the differential $df$ vanishes.  There may be points where the differential has rank $1,$ but these may also be understood via this holomorphic picture.  Let $z=x+iy$ be a local isothermal coordinate.  Then a local holomorphic frame of $K^{-1}$ is given by
\begin{align}
\frac{\partial}{\partial z}=\frac{1}{2}\left(\frac{\partial}{\partial x}-i\frac{\partial}{\partial y}\right).
\end{align}
Let $\Psi: TX \rightarrow K^{-1}$ be the holomorphic isomorphism defined via,
\begin{align}
\Psi\left(\frac{\partial}{\partial x}\right)=\frac{\partial}{\partial z},
\end{align}
and
\begin{align}
\Psi\left(\frac{\partial}{\partial y}\right)=i \frac{\partial}{\partial z}.
\end{align}
Lastly, if $\iota: f^{*}(TM)\rightarrow f^{*}(TM)\otimes \mathbb{C}$ is the natural inclusion, then the following identities hold:
\begin{align}
\iota\circ df\left(\frac{\partial}{\partial x}\right)=2\text{\textnormal{Re}}\left(\partial f\circ \Psi\left(\frac{\partial}{\partial x}\right)\right),
\end{align}
and
\begin{align}
\iota\circ df\left(\frac{\partial}{\partial y}\right)=2\text{\textnormal{Re}}\left(\partial f\circ \Psi\left(\frac{\partial}{\partial y}\right)\right).
\end{align}
Putting the previous discussion together, and noting that every non-zero tangent vector appears as some isothermal coordinate vector, we deduce the following.
\begin{proposition}\label{immersion}
A harmonic map $f:X\rightarrow M$ has no branch points if and only if $\partial f$ is injective if and only if $L\simeq K^{-1}.$  Furthermore, $f$ is an immersion if and only if for all choices of isothermal coordinates,
\begin{align}
\iota\circ df\left(\frac{\partial}{\partial x}\right)=2\text{\textnormal{Re}}\left(\partial f\circ \Psi\left(\frac{\partial}{\partial x}\right)\right),
\end{align}
and,
\begin{align}
\iota\circ df\left(\frac{\partial}{\partial y}\right)=2\text{\textnormal{Re}}\left(\partial f\circ \Psi\left(\frac{\partial}{\partial y}\right)\right),
\end{align}
are never zero.
\end{proposition}

A further holomorphic consequence of harmonicity is the following proposition due to Hopf \cite{HOP54},
\begin{proposition} \label{hopfdiff}
Let $f:X\rightarrow M$ be a harmonic map.  Then the expression
\begin{align}
\alpha=h\left(df\left(\frac{\partial}{\partial x}\right),df\left(\frac{\partial}{\partial x}\right)\right)-h\left(df\left(\frac{\partial}{\partial y}\right),df\left(\frac{\partial}{\partial y}\right)\right) \\ -\ 2ih\left(df\left(\frac{\partial}{\partial x}\right),df\left(\frac{\partial}{\partial y}\right)\right) dz^2,
\end{align}
defines a holomorphic quadratic differential on $X.$  
\end{proposition}
The harmonic map is \textit{weakly} conformal if and only if $\alpha=0.$  This is equivalent \cite{SU82} to $f:X\rightarrow M$ being a branched minimal immersion.  In particular, the expression for the Hopf differential $\alpha$ makes clear the differential of a weakly conformal harmonic map can only have rank $0$ or rank $2.$  In light of Proposition \ref{immersion}, a minimal surface is without branch points if and only if the $(1,0)$-part of the differential is an injective bundle mapping.

\section{Hitchin Representations}

\subsection{Principal $3$-dimensional subalgebra}
The following is a rapid review of the theory of complex simple Lie algebras, Higgs bundles and Hitchin representations.  See \cite{HIT92} for a thorough treatment.  Fix a complex adjoint simple Lie group G.  Recall that the classical groups are $\text{PSL}_n(\mathbb{C}), \text{PSO}_{2n}(\mathbb{C}),\\ \text{PSO}_{2n+1}(\mathbb{C}),$ and $\text{PSp}_{2n}(\mathbb{C})).$  If $\mathfrak{g}$ is the Lie algebra of $G,$ Lie theory establishes the existence of a $3$-dimensional Lie subalgebra $\mathfrak{sl}(2,\mathbb{C}) \subset \mathfrak{g}$ called the principal $3$-dimensional subalgebra, which is a copy of the Lie algebra of $\text{SL}_{2}(\mathbb{C})$ inside $\mathfrak{g}.$  Exponentiating one obtains an irreducible representation,
\begin{align}
\text{\textnormal{PSL}}_{2}(\mathbb{C})\rightarrow G.
\end{align}
The principal $3$-dimensional subalgebra is generated by elements $x, e_{-1}, e_{1}\in \mathfrak{g}$ satisfying the relations,
\begin{align}
[x,e_{-1}]&=-e_{-1}, \\
[x,e_{1}]&=e_{1}, \\
[e_{-1},e_{1}]&=x.
\end{align}
If $\ell\geq 1$ is the rank of $\mathfrak{g},$ then there exist an additional $\ell -1$ elements $e_2,...,e_{\ell}\in \mathfrak{g}$ which, together with $e_1,$ span the centralizer of the adjoint action of $e_{1}$ on $\mathfrak{g}.$  These $e_i$ are also the highest weight vectors for the adjoint action of the principal $3$-dimensional subalgebra on $\mathfrak{g}.$  Restricting to the adjoint action of the semi-simple element $x$ yields an eigen-decomposition:
\begin{align}\label{weightdecomp}
\mathfrak{g}=\bigoplus_{i=-M}^{M}\mathfrak{g}_i,
\end{align}
with $e_i\in \mathfrak{g_{m_i}}.$  The numbers $1=m_1<m_2<....<m_{\ell}=M$ are the exponents of $\mathfrak{g}.$ 

\subsection{Higgs bundles}
Let $G$ be a complex semi-simple Lie group with maximal compact subgroup $K\subset G,$ and $X$ a compact Riemann surface of genus at least two.
\begin{definition}
A $G$-Higgs bundle is a triple $(P,A,\phi)$ where $P$ is a principal $K$-bundle over $X,\ A$ is a connection on $P,$ and $\phi$ (the \textit{Higgs} field) is a holomorphic section of the twisted complex adjoint bundle $K\otimes(\text{\textnormal{Ad}}(P)\otimes \mathbb{C}).$
\end{definition}
Above,
\begin{align}
\text{\textnormal{Ad}}(P)=P\times_{Ad(K)} \mathfrak{k},
\end{align}
denotes the adjoint bundle of Lie algebras with $\mathfrak{k}$ denoting the Lie algebra of $K.$  The holomorphic structure on $\text{Ad}(P)\otimes \mathbb{C}$ is that induced from the $K$-connection $A$ on $P.$  Bracket with $\phi$ defines an endomorphism of $\text{Ad}(P)\otimes \mathbb{C}$ with values in the canonical bundle $K,$ yielding the confluent definition.
\begin{definition}
A Higgs (vector) bundle is a pair $(E,\phi)$ where $E$ is a holomorphic vector bundle over $X$ and the Higgs field $\phi$ is a holomorphic section of $K\otimes \text{\textnormal{End}}(E).$
\end{definition}
A Higgs (vector) bundle $(E,\phi)$ is said to be stable if and only if for every $\phi$-invariant, holomorphic sub-bundle $F\subset E,$ the following inequality is satisfied:
\begin{align}
\frac{\text{deg}(F)}{\text{rk}(F)}<\frac{\text{deg}(E)}{\text{rk}(E)},
\end{align}
where the numerator is the degree of the holomorphic bundle and the denominator is the rank.  We note that there is a more subtle definition of a stable $G$-Higgs bundle, but what we have given here is suitable for our purposes.  A Higgs bundle is poly-stable if it is a direct sum of stable Higgs bundles.  Since this will be the case in our applications, we assume from here forward that the vector bundle $\text{Ad}(P)\otimes \mathbb{C}$ is degree zero.  

The Hitchin-Kobayashi correspondence (due in this case to Hitchin \cite{HIT87} and Simpson \cite{SIM92}) relates poly-stable Higgs bundles to solutions of an important gauge-theoretic system of equations called the self-duality equations.  Suppose we are given a holomorphic structure on $\text{Ad}(P)\otimes \mathbb{C}$ and a holomorphic section $\phi\in K\otimes (\text{Ad}(P)\otimes \mathbb{C})$ acting as a $K$-valued endomorphism via the Lie bracket.  
If this pair is poly-stable in the above sense, then there exists a unique Hermitian metric $H$ on $\text{Ad}(P)\otimes \mathbb{C}$ yielding Chern connection $\nabla^{H}$ such that
the self-duality equations are satisfied:
\begin{align}\label{SD equations}
F(\nabla^{H})+[\phi,\phi^{*_H}]=0, \\
\overline{\partial}_{H}\phi=0.
\end{align}
In the first equation, $F$ is the curvature of the Chern connection and the adjoint $\phi^{*_H}$ is computed using the Hermitian metric $H.$  The second equation is immediate since $\phi$ was assumed holomorphic and the Chern connection induces the existing holomorphic sturcture.  From \eqref{SD equations}, a short computation reveals that the $G$-connection
\begin{align}
B=\nabla^{H}+\phi+\phi^{*_{H}},
\end{align}
is flat.  Thus, the holonomy of this flat connection defines a representation
\begin{align}
\rho:\pi_1(\Sigma)\rightarrow G,
\end{align}
which is reductive, namely that the induced action on the Lie algebra of $G$ is completely reducible.

Conversely, start with a flat principal $G$-bundle $(P,B)$ where $B$ is the flat connection over the Riemann surface $X.$  Since $G/K$ is contractible, we may reduce the structure group from $G$ to $K$ which correspondingly yields a decomposition of the flat connection, 
\begin{align}
B=A+\Psi,
\end{align}
where $A$ is a $K$-connection and $\Psi$ is a $1$-form with values in the adjoint bundle.
Furthermore, we may split $\Psi$ into $(1,0)$ and $(0,1)$ parts,
\begin{align} \label{phidefine}
\Psi=\phi+\phi^{*_H}.
\end{align}
The first equation from \eqref{SD equations} is automatic from the flatness of $B,$ whereas the second equation requires that the reduction of structure group is of a special type.  The flat $G$-bundle can be reconstructed as a fibered product of the universal cover $\widetilde{X}$ and $G,$
\begin{align}
P=\widetilde{X}\times_{\rho} G,
\end{align}
where $\rho$ is the monodromy of $B,$ and the reduction of structure group to $K$ is equivalent to a $\rho$-equivariant map
\begin{align}
f: \widetilde{X}\rightarrow G/K.
\end{align}
Unwinding the definitions reveals that the second equation of \eqref{SD equations} is satisfied if and only if the $\rho$-equivariant map is harmonic in the Riemannian geometric sense discussed earlier in this paper.  The $(1,0)$-part of the differential of $f$ is, after a suitable identification, the section $\phi$ defined in \eqref{phidefine}.  A theorem of Corlette \cite{COR88}, due originally in the $\text{PSL}_2(\mathbb{C})$ case to Donaldson \cite{DON87}, states that such an equivariant harmonic map exists if and only if the connection $B$ is reductive, namely that any $B$-invariant sub-bundle of $P$ has a $B$-invariant complement.  A vanishing theorem shows that the corresponding Higgs bundle with the holomorphic structure being given by the connection $A$ is poly-stable, and furthermore stable if and only if the flat connection $B$ is irreducible.  

The group $G$ is acting by conjugation on the set of reductive representations $\pi_1(\Sigma)\rightarrow G,$ which translates to an action on Higgs bundles via bundle automorphisms.  The above reversible constructions yield a homeomorphism between the space of conjugacy classes of reductive representations $\pi_1(\Sigma)\rightarrow G,$ and the space of gauge equivalence classes of poly-stable Higgs bundles; this homeomorphism is know as the \textit{non-abelian Hodge correspondence} since it is a higher rank analogue of Hodge theory for holomorphic line bundles over X (see \cite{GX08}).

\subsection{Hitchin representations}
As detailed above, Higgs bundles can be seen as a machine for producing representations $\pi(\Sigma)\rightarrow G.$  Hitchin used this technique to isolate an important component of representations $\pi_1(\Sigma)\rightarrow G^{r}$ where $G^{r}$ is the split real form of a complex adjoint simple Lie group $G$ of rank $\ell.$  The split real forms of the classical linear groups consist of $\text{PSL}_n(\mathbb{R}), \newline \text{PSp}_{2n}(\mathbb{R}), \text{PSO}(n,n),$ and $\text{PSO}(n,n+1).$  Generally, there is a unique (up to the adjoint action) real form of $\mathfrak{g}$ for which the Killing form has maximal index and the exponentiation of this sub-algebra is the split real form. We quickly review Hitchin's construction.  Recall the weight decomposition \eqref{weightdecomp} of the Lie algebra of $G,$
\begin{align} 
\mathfrak{g}=\bigoplus_{i=-M}^{M}\mathfrak{g}_i.
\end{align}

  Consider the holomorphic vector bundle,
\begin{align}
E=\bigoplus_{i=-M}^{M} \underline{\mathfrak{g}_i}\otimes K^i.
\end{align}
Here, $\underline{\mathfrak{g}_i}$ notates the trivial vector bundle over the closed Riemann surface $X$ with fiber $\mathfrak{g}_i$ and $K$ is the canonical bundle of holomorphic $1$-forms.  Note that this is a holomorphic vector bundle with structure group $G.$  Next, consider the expression,
\begin{align}
\phi=e_{-1}+\sum_{k=1}^{\ell} \alpha_k \otimes e_k,
\end{align}
where $\alpha_k$ is a holomorphic section of $K^{m_k +1}.$  Firstly, $\phi$ is evidently a holomorphic section of $K\otimes E.$  Next, $\phi$ acts as an endomorphism with values in $K$ by combining the tensor product of holomorphic differentials with the Lie bracket on $\mathfrak{g}.$  As observed by Hitchin \cite{HIT92}, the bracket operation is holomorphic, thus we get a Higgs (vector) bundle $(E,\phi),$ with $\phi\in H^0(K\otimes \text{End}(E)).$  Hitchin \cite{HIT92} showed that these Higgs bundles are all stable, and hence by the non-abelian Hodge correspondence yield irreducible representations of the fundamental group $\pi_1(\Sigma)\rightarrow G,$ which moreover actually take values in the split real form $G^r$ of $G.$  Furthermore, these representations are parameterized by the \textit{Hitchin base} $\bigoplus_{i=1}^{\ell}K^{m_i+1}.$  The totality of all such representations forms a component of the space of conjugacy classes of representations $\pi_1(\Sigma)\rightarrow G^{r},$ which is called the Hitchin component.  We will denote the Hitchin component by $\text{Hit}(G).$  For us, the useful thing is that $\phi$ is an expression for the $(1,0)$-part of the differential of the corresponding equivariant harmonic map which is known to exist by the theorem of Corlette \cite{COR88}.  Therefore, we may use the discussion in section \ref{harmmaps} to obtain properties of the differential of those harmonic maps.

\section{Hitchin harmonic maps are immersions}

Let $\rho\in \text{Hit}(G)$ for $G$ a complex adjoint simple Lie group of rank $\ell.$  Let $X$ be a compact Riemann surface of genus greater than one with universal cover $\widetilde{X}.$  Then by Corlette \cite{COR88}, there exists a unique $\rho$-equivariant harmonic map
\begin{align}
f_{\rho}:\widetilde{X}\rightarrow G/K.
\end{align}
We now review the relationship between the adjoint bundle description in the previous section, and the tangent bundle of the symmetric space $G/K.$  From $\rho,$ we first construct the flat principal $G$-bundle
\begin{align}
P=\widetilde{X}\times_{\rho} G.
\end{align}
The harmonic map $f_{\rho}$ is the pullback under the universal covering map $p:\widetilde{X}\rightarrow X$ of a \textit{harmonic} section $s_{\rho}$ of the associated bundle of symmetric spaces,
\begin{align}
P\times_{G} G/K \simeq \widetilde{X}\times_{\rho} G/K\xrightarrow{\pi} X.
\end{align}
This section defines a reduction of structure group of $P$ from $G$ to $K$ and hence gives a principal $K$-bundle $P_K$ equipped with a $K$-connection, the latter being the pullback of the part of the flat connection taking values in the Lie algebra of $K.$

Next, taking derivatives yields a sequence of bundles,
\begin{align}
0\rightarrow \mathcal{V}=\text{ker}(d\pi)\rightarrow T(\widetilde{X}\times_{\rho} G/K)\rightarrow \widetilde{X}\times_{\rho} G/K \rightarrow 0.
\end{align}
Pulling back the bundle $\mathcal{V}$ to a bundle over $X$ using the section $s_{\rho}$ gives
\begin{align}
s_{\rho}^{*}(\mathcal{V})\rightarrow X.
\end{align}
We note that $\mathcal{V}$ is a bundle with fiber given by the fibers of the tangent bundle $T(G/K)$ to the symmetric space $G/K$ which is also equipped with a $K$-connection.  Recalling the Cartan decomposition $\mathfrak{g}=\mathfrak{k}\oplus \mathfrak{p},$ this tangent bundle is $G$-equivariantly,
\begin{align}
T(G/K)\simeq G\times_{Ad(K)} \mathfrak{p}.
\end{align}
On the other hand, the adjoint bundle
\begin{align}
Ad(P_K)\rightarrow X,
\end{align}
is a bundle with fiber $\mathfrak{k}.$  Complexifying yields an isomorphism,
\begin{align}
s_{\rho}^{*}(\mathcal{V})\otimes \mathbb{C}\simeq Ad(P_K)\otimes \mathbb{C},
\end{align}
as bundles equipped with connection, and hence as holomorphic bundles over $X.$  Lastly, we observe that the complexified adjoint bundle is none other than the bundle $E$ from our Higgs bundle description,
\begin{align}
Ad(P_K)\otimes \mathbb{C}\simeq  E=\bigoplus_{i=-M}^{M} \underline{\mathfrak{g}_i}\otimes K^i.
\end{align}

Now, we relate this to the equivariant harmonic map $f_{\rho}.$  The tangent bundle $T(G/K)$ is an equivariant vector bundle over $G/K$ for the action of $G$ by left translation.  Since $f_{\rho}$ is $\rho$-equivariant, the pull-back bundle $f_{\rho}^{*}(T(G/K))\otimes \mathbb{C}$ becomes a $\pi_1(X)$-equivariant vector bundle over $\widetilde{X}.$  This yields the identification,
\begin{align}\label{equi-iso}
\pi_1(X)\backslash (f_{\rho}^{*}(T(G/K))\otimes \mathbb{C})\simeq s_{\rho}^{*}(\mathcal{V})\otimes \mathbb{C}.
\end{align}
Remarking that all of the connections present on the bundles pullback as well, which induce holomorphic structures, the following proposition holds.
\begin{proposition}\label{vbiso}
There is an isomorphism,
\begin{align}
\beta: \pi_1(X)\backslash (f_{\rho}^{*}(T(G/K))\otimes \mathbb{C})\simeq Ad(P_{K})\otimes\mathbb{C}\simeq E=\bigoplus_{i=-M}^{M} \underline{\mathfrak{g}_i}\otimes K^i,
\end{align}
as holomorphic vector bundles.
\end{proposition}

\textbf{Remark:}  At this point, it is unclear to the author if one can easily show directly from properties of the representation and the harmonic map that this \textit{is} the holomorphic structure on the adjoint bundle.  Our current proof which starts with the Higgs bundle and then works through the non-abelian Hodge correspondence certainly leaves something to be desired.

Using this isomorphism we obtain the following commutative diagram:
\begin{center}
\begin{tikzpicture}[column sep=1in,row sep=1in]
    \matrix (A) [matrix of math nodes]
    {
       \partial f_{\rho}: p^{*}(K^{-1})& f_{\rho}^{*}(T(G/K))\otimes \mathbb{C}\\
       \phi: K^{-1}& E,\\
    };
   \draw[->] (A-1-1) -- (A-1-2);
    \draw[->] (A-1-1) -- (A-2-1);
    \draw[->] (A-2-1) -- (A-2-2);
    \draw[->] (A-1-2) -- (A-2-2);
\end{tikzpicture}
\end{center}
where the left vertical arrow is the quotient map under the $\pi_1(X)$-action and the right vertical arrow, which we will denote by $\overline{\beta},$ is the isomorphism $\beta$ of proposition \ref{vbiso} precomposed with the quotient by the $\pi_1(X)$-action.  We remark that the top line consists of $\pi_1(X)$-equivariant holomorphic bundles over $\widetilde{X}$ and the bottom line consists of holomorphic vector bundles over $X.$  Since the vertical arrows are local diffeomorphisms, the mapping properties of $\partial f_{\rho}$ are determined by that of $\phi.$  

By the Hitchin parameterization, 
\begin{align}
\phi=e_{-1}+\sum_{k=1}^{\ell} \alpha_k \otimes e_k,
\end{align}
for some choice of pluri-canonical sections $\alpha_{k}\in \text{H}^{0}(K^{m_k+1})$ where the Higgs bundle $(E,\phi)$ corresponds, via the non-Abelian Hodge correspondence, to $\rho.$  How does $\phi$ act as such a map?  Pick a local holomorphic frame $\frac{\partial}{\partial z}$ of $K^{-1},$ then contract each pluri-canonical section appearing in $\phi$ with this this vector to obtain:
\begin{align}
\phi\left(\frac{\partial}{\partial z}\right)=\frac{\partial}{\partial z} \otimes e_{-1}+\sum_{k=1}^{\ell} \alpha_k\left(\frac{\partial}{\partial z},-\right) \otimes e_k \in E.
\end{align}
Remembering our earlier discussion of harmonic maps from section \ref{harmmaps}, let $\iota:f_{\rho}^{*}(T(G/K))\rightarrow f_{\rho}^{*}(T(G/K))\otimes \mathbb{C}$ be the natural inclusion and $\Psi:p^{*}T\Sigma \rightarrow p^{*}K^{-1}$ the holomorphic identification of the tangent bundle with the holomorphic tangent bundle. Then,
\begin{align}\label{dfdx}
\overline{\beta}\circ\iota\circ df_{\rho}\left(\frac{\partial}{\partial x}\right)&=2\text{\textnormal{Re}}\left(\overline{\beta}\circ\partial f_{\rho}\circ \Psi\left(\frac{\partial}{\partial x}\right)\right)\\
&=2\text{\textnormal{Re}}\left(\phi\left(\frac{\partial}{\partial z}\right)\right),
\end{align}
and,
\begin{align}\label{dfdx'}
\overline{\beta}\circ\iota\circ df_{\rho}\left(\frac{\partial}{\partial y}\right)&=2\text{\textnormal{Re}}\left(\overline{\beta}\circ\partial f_{\rho}\circ \Psi\left(\frac{\partial}{\partial y}\right)\right) \\
&=2\text{\textnormal{Re}}\left(\phi\left(i\frac{\partial}{\partial z}\right)\right).
\end{align}
For any choice of isothermal coordinates, $\phi$ acts according to,
\begin{align}
\text{\textnormal{Re}}\left(\phi\left(\frac{\partial}{\partial z}\right)\right)=\frac{1}{2}\frac{\partial}{\partial x}\otimes e_{-1}+\text{\textnormal{Re}}\left(\sum_{k=1}^{\ell} \alpha_k\left(\frac{\partial}{\partial z},-\right) \otimes e_k\right),
\end{align}
and,
\begin{align}
\text{\textnormal{Re}}\left(\phi\left(i\frac{\partial}{\partial z}\right)\right)=\frac{1}{2}\frac{\partial}{\partial y}\otimes e_{-1}+\text{\textnormal{Re}}\left(\sum_{k=1}^{\ell} \alpha_k\left(i\frac{\partial}{\partial z},-\right) \otimes e_k\right).
\end{align}
Since the sum $E=\bigoplus_{i=-M}^{M} \underline{\mathfrak{g}_i}\otimes K^i$ is direct, the non-vanishing of the $K^{-1}\otimes \underline{\mathfrak{g}_{-1}}$ piece implies that both,
\begin{align}
\text{\textnormal{Re}}\left(\phi\left(\frac{\partial}{\partial z}\right)\right),
\end{align}
and
\begin{align}
\text{\textnormal{Re}}\left(\phi\left(i\frac{\partial}{\partial z}\right)\right),
\end{align}
never vanish. 
Hence, from \eqref{dfdx} and \eqref{dfdx'}, 
\begin{align}
df_{\rho}\left(\frac{\partial}{\partial x}\right),\ df_{\rho}\left(\frac{\partial}{\partial y}\right)\neq 0.
\end{align}
Having seen Proposition \ref{immersion} in action, this proves $df_{\rho}$ is injective.
\begin{theorem}\label{Main}
Let $G$ be a complex adjoint simple Lie group, $\rho\in \text{Hit}(G)$ be a Hitchin representation, and $f_{\rho}:\widetilde{X}\rightarrow G/K$ the corresponding $\rho$-equivariant harmonic map.  Then $f_{\rho}$ is an immersion.
\end{theorem}

We remark that this theorem follows from the result of Schoen-Yau \cite{SY78} in the case that $G=\text{PSL}_2({\mathbb{C}}).$  Next, if one is only interested in the above statement for those harmonic maps which are also conformal, namely those which are branched minimal immersions, corresponding to $\alpha_1=0,$ the proof simplifies.  For it is immediate that the Higgs field $\phi$ is a non-vanishing map of holomorphic vector bundles.  Then, it follows from the comment following Proposition \ref{hopfdiff} that all such maps are immersions and the more detailed argument given above is unnecessary.  

In rank one, the theorem of Schoen-Yau asserts that the harmonic map is an \textit{embedding}.  It is an interesting problem, which we state as a conjecture here, to explore this in higher rank.

\begin{conjecture}
Let $G$ be a complex adjoint simple Lie group, $\rho\in \text{Hit}(G)$ be a Hitchin representation, and $f_{\rho}:\widetilde{X}\rightarrow G/K$ the corresponding $\rho$-equivariant harmonic map.  Then $f_{\rho}$ is an embedding.
\end{conjecture}

\section{Applications}
In this section, we record some interesting consequences of Theorem \ref{Main}.
\subsection{Volume entropy}
We start this section with a review of some basic submanifold theory.   Suppose $M$ is a complete, Riemannian manifold of dimension at least $3$ and $f:\widetilde{\Sigma}\rightarrow M$ is an immersion.  Let $\nabla$ denote the Levi-Civita connection of $M$ and $\nabla^{T}$ the component of $\nabla$ tangential to the image of $f.$  Then the second fundamental form is the symmetric 2-tensor with values in the normal bundle given by,
\begin{align}
B(X,Y)=\nabla_{X} Y-\nabla_{X}^{T} Y,
\end{align}
where $X, Y\in \Gamma(f^{*}TM)$ are tangent to the image of $f.$  Thus, $B$ is a normal bundle valued symmetric $2$-tensor, which we write in a local trivialization of the normal bundle as $(B_1,...,B_{n-2}).$  Then each $B_i$ is an ordinary real valued $2$-tensor over the chosen trivializing open set.  Denoting by $g$ the induced metric of the immersion $f,$ the immersion is \textit{minimal} if the trace of the $B_i$ with respect to $g$ simultaneously vanish:
\begin{align}
\text{tr}_g{B_i}=H_i=0.
\end{align}
For each $i,\ H_i$ is the i-th mean curvature of the immersion.  Denoting the Riemannian metric on $M$ by angled brackets $\langle \ ,\ \rangle,$ the Gauss equation reads:
\begin{align}
K_g=\text{Sec}(\partial_1, \partial_2)-\langle B(\partial_1,\partial_2), B(\partial_1,\partial_2)\rangle +\langle B(\partial_1,\partial_1), B(\partial_2,\partial_2)\rangle,
\end{align}
where $K_g$ is the sectional curvature of $g$ and $\text{Sec}(\partial_1,\partial_2)$ is the sectional curvature of the two plane spanned by $\{\partial_1,\partial_2\}$ computed in the Riemannian metric of $M$.
Choosing isothermal coordinates on $\widetilde{\Sigma}$ for the metric $g$ and writing the result with respect to an orthonormal framing of the normal bundle, the minimality of $f$ implies,
\begin{align}
B(\partial_1,\partial_1)=-B(\partial_2,\partial_2).
\end{align}
This verifies that the Gauss equation in this setting is,
\begin{align}
K_{g}=\text{Sec}(\partial_1,\partial_2)-\frac{1}{2}\lVert B\rVert_{g}^2.
\end{align}

Next we need to introduce an important inequality due to Manning \cite{MAN81}.  Suppose $(\Sigma,g)$ is a Riemannian surface of genus at least $2.$.  The unit tangent bundle to $\Sigma$ admits a measure $m_{L},$ the \textit{Liouville measure}, which is invariant under the geodesic flow; in a local trivialization $m_{L}$ is a product of Riemannian volume on $\Sigma$ with standard Lebesgue measure on the circle.  We normalize so that $m_{L}$ has unit mass.  To any invariant probability measure on the unit tangent bundle $T^{1}(\Sigma),$ there is an associated non-negative real number $h(m_{L})$ called the measure-theoretic entropy, which is a measure of the dynamical/measure complexity of the geodesic flow with respect to $m_{L}$ (see \cite{HK95} for a discussion of these objects).  In \cite{MAN81}, Manning proved the following inequality.
\begin{theorem}[\cite{MAN81}]\label{thm: manningbound}
Let $(\Sigma,g)$ be a Riemannian surface of non-positive curvature.  Then,
\[\frac{1}{\sqrt{\text{Vol}(g)}}\int_{\Sigma}\sqrt{-K_g}\ dV_{g}\leq h(m_{L}),\]
where $h(m_{L})$ is the measure theoretic entropy of the geodesic flow with respect to Liouville measure.
\end{theorem}
\textbf{Remark:} Note that if $K=-1$ in the above formula, the inequality becomes equality: $h(m_L)=\sqrt{\text{Vol}(g)}=\sqrt{-2\pi \chi(\Sigma)}.$

Somewhat earlier, Manning \cite{MAN79} introduced the following important definition.
\begin{definition}
Let $(M,g)$ be a closed, oriented Riemannian manifold.  Let $B_{g}(p,R)$ be the metric ball in the universal cover $\widetilde{M}$ of radius $R$ centered at a basepoint $p\in \widetilde{M}.$  The volume entropy is the quantity,
\[E(g)=\displaystyle\limsup_{R\rightarrow\infty}\frac{\log\lvert B_g(p,R)\rvert}{R},\]
where $\lvert B_g(p,R)\rvert$ is the Riemannian volume of the ball centered at $p$ of radius $R.$
\end{definition}
In \cite{SAN14}, we prove the following simple corollary to Theorem \ref{thm: manningbound}.
\begin{theorem} \label{thm: entbound}
Suppose $(\Sigma,g)$ is a Riemannian surface of genus at least two and $g$ has non-positive sectional curvature $K_g.$ Then the volume entropy satisfies,
\begin{align}
\frac{1}{\text{Vol}(g)}\int_{\Sigma} \sqrt{-K_g}\ dV_{g}\leq E(g).
\end{align}
\end{theorem}

Now, let $\rho:\pi_1(\Sigma)\rightarrow G$ be a Hitchin representation and select a point in the symmetric space $p\in G/K.$  Define a function,
\begin{align}
N_{p,\rho}(R):=|\{\gamma\in \pi_1(\Sigma) \ | \ d(p,\rho(\gamma)(p))\leq R\}|,
\end{align}
where the distance is measured in the symmetric space $G/K.$

\begin{definition}
The volume entropy of $\rho$ based at $p\in G/K$ is,
\begin{align}
h(\rho,p):=\limsup_{R\rightarrow \infty} \frac{\log(N_{p,\rho}(R))}{R}.
\end{align}
\end{definition}

\textbf{Remark:}  It follows readily from the triangle inequality that the volume entropy is independent of the point $p\in G/K,$ hence we will simply refer to \textit{the} volume entropy $h(\rho).$

The following theorem yields a lower bound on $h(\rho).$

\begin{theorem}\label{entbound}
Let $\rho:\pi_{1}(\Sigma)\rightarrow G$ be a Hitchin representation where $G$ is a complex adjoint simple Lie group $G.$  Then there exists,
\begin{align}
f_{\rho}:\widetilde{\Sigma}\rightarrow G/K,
\end{align}  
a $\rho$-equivariant minimal immersion.  Denote by $(g,B)$ the induced metric and second fundamental form of the minimal immersion.
Let $h(\rho)$ be the volume entropy of $\rho.$  Then,
\begin{align}
\frac{1}{\text{\textnormal{Vol}}(g)}\int_{\Sigma}\sqrt{-\text{\textnormal{Sec}}\left(T_{f_{\rho}}(f_{\rho}(\widetilde{\Sigma}))\right)+\frac{1}{2}\lVert B\rVert_{g}^2} \ dV_{g} \leq h(\rho).
\end{align}
\end{theorem}

\begin{proof}
Let $\rho:\pi_1(\Sigma)\rightarrow G$ be a Hitchin representation, by a theorem of Labourie \cite{LAB08} there exists a $\rho$-equivariant branched minimal immersion,
\begin{align}
f_{\rho}:\widetilde{\Sigma}\rightarrow G/K.
\end{align}
Since minimal maps are equivalently conformal, harmonic maps \cite{SU82}, Theorem \ref{Main} implies $f_{\rho}$ is an immersion.  Let $g$ be the induced metric via the immersion $f_{\rho}.$  Fixing $p\in\widetilde{\Sigma},$ define the function,
\begin{align}
\widetilde{N}_{p}(R):=\lvert \{\gamma \in \pi_1(\Sigma) \ | \ d_{g}(p,\gamma(p))\leq R\} \rvert.
\end{align}
 Because $\rho$ is a Hitchin representation, it is faithful \cite{LAB06}.  Since $f_{\rho}$ is an isometric immersion, it follows that,
\begin{align}
\widetilde{N}_{p}(R)\leq N_{f_{\rho}(p),\rho}(R).
\end{align}
This inequality is just expressing the fact that distances measured in $g$ dominate distances measured in the symmetric space.
Taking logarithms of each side, dividing by $R,$ and taking the $\limsup$ reveals,
\begin{align} \label{growthbound}
\limsup_{R\rightarrow \infty}\frac{\widetilde{N}_{p}(R)}{R}\leq \limsup_{R\rightarrow \infty}\frac{N_{f_{\rho}(p),\rho}(R)}{R}.
\end{align}
It is a straightforward exercise to see that the left hand side converges to $E(g)$ and the right hand side converges to $h(\rho,f_{\rho}(p))$ by definition.  This implies,
\begin{align}
E(g)\leq h(\rho).
\end{align}
Since $f$ is minimal, the Gauss equation reads,
\begin{align}
K_{g}=\text{\textnormal{Sec}}(T_{f_{\rho}}(f_{\rho}(\widetilde{\Sigma})))-\frac{1}{2}\lVert B\rVert_{g}^2.
\end{align}
Since $G/K$ is a non-compact symmetric space, it has non-positive sectional curvature from which it follows that $g$ has non-positive sectional curvature.  Applying Theorem \ref{entbound},
\begin{align}
\frac{1}{\text{Vol}(g)}\int_{\Sigma} \sqrt{-K_g}\ dV_{g}\leq E(g)\leq h(\rho,f_{\rho}(p)).
\end{align}
Plugging in the Gauss equation gives the desired estimate,
\begin{align}
\frac{1}{\text{Vol}(g)}\int_{\Sigma} \sqrt{-\text{\textnormal{Sec}}(T_{f_{\rho}}(f_{\rho}(\widetilde{\Sigma})))+\frac{1}{2}\lVert B\rVert_{g}^2}\ dV_{g}\leq h(\rho).
\end{align}
\end{proof}
This inequality gives a cheap way to see that that the volume entropy of $\rho$ is positive.  Employing thermodynamic formalism, much more refined control on this quantity is obtained by Sambarino \cite{SAM12}.
\begin{corollary}
With the notation above, $h(\rho)>0.$
\end{corollary}
\begin{proof}
Suppose $h(\rho)=0.$ Then Theorem \ref{entbound} implies that for any $\rho$-equivariant, immersed minimal surface with induced metric and second fundamental form $(g,B),$ the following equality holds:
\begin{align}
\frac{1}{\text{\textnormal{Vol}}(g)}\int_{\Sigma}\sqrt{-\text{\textnormal{Sec}}(T_{f_{\rho}}(f_{\rho}(\widetilde{\Sigma})))+\frac{1}{2}\lVert B\rVert_{g}^2} \ dV_{g}=0.
\end{align}
This implies that the integrand vanishes,
\begin{align} \label{equal 0}
\sqrt{-\text{\textnormal{Sec}}(T_{f_{\rho}}(f_{\rho}(\widetilde{\Sigma})))+\frac{1}{2}\lVert B\rVert_{g}^2}=0.
\end{align}
Since $G/K$ has non-positive sectional curvature, \eqref{equal 0} implies, 
\begin{align}
\text{\textnormal{Sec}}(T_{f_{\rho}}(f_{\rho}(\widetilde{\Sigma})))=\lVert B\rVert_{g}^2=0.
\end{align}
By the Gauss equation, $K_{g}=\text{\textnormal{Sec}}(T_{f_{\rho}}(f_{\rho}(\widetilde{\Sigma})))-\frac{1}{2}\lVert B\rVert_{g}^2=0.$  Since $(\Sigma,g)$ is a metric on a surface of genus at least two, this is impossible by the Gauss-Bonnet theorem.  This completes the proof.
\end{proof}
The following Corollary is immediate from Theorem \ref{entbound}.  
\begin{corollary}
Let $\rho_n:\pi_{1}(\Sigma)\rightarrow G$ be a sequence of Hitchin representations with corresponding sequence of $\rho_n$-equivariant minimal surfaces,
\begin{align}
f_{\rho_n}:\widetilde{\Sigma}\rightarrow G/K.
\end{align}
Suppose $h(\rho_n)\rightarrow 0.$
Then,
\begin{align}
\frac{1}{\text{\textnormal{Vol}}(g_n)}\int_{\Sigma}\sqrt{-\text{\textnormal{Sec}}\left(T_{f_{\rho_n}}(f_{\rho_n}(\widetilde{\Sigma}))\right)+\frac{1}{2}\lVert B_n\rVert_{g_n}^2} \ dV_{g_n} \rightarrow 0.
\end{align}
\end{corollary}

\textbf{Remark:}  The mantra of this corollary is: nearly zero volume entropy implies nearly flat, nearly totally geodesic equivariant minimal surfaces.  We make the important remark that \textit{nearly} in the previous sentence is only meant to be interpreted in a measure theoretic sense: it is certainly possible, and actually expected, that there exists points on the surface where the curvature concentrates and blows up in the limit.

The above corollary is only useful if we can produce sequences of Hitchin representations satisfying the hypothesis that $h(\rho_n)\rightarrow 0.$  Some very recent progress has been made in this direction for $\text{SL}_3(\mathbb{C})$-Hitchin representations by Zhang \cite{ZHA13} who exhibits certain sequences of Hitchin representations along which $h(\rho)\rightarrow 0.$  In another direction, a recent preprint of Collier-Li \cite{CL14} exhibits sequences of $\text{SL}_n(\mathbb{C})$-Hitchin representations for which the equivariant minimal surfaces become increasingly flat in the induced metric.  Their methods, which entail obtaining precise asymptotics on the Harmonic metric associated to certain one parameter families of Hitchin representations, are totally different than those present here.  It is unknown if the sequences $\rho_n:\pi_1(\Sigma)\rightarrow \text{SL}_n(\mathbb{C})$ which they study satisfy $h(\rho_n)\rightarrow 0.$ 

The problem of understanding the formal mathematical objects encoded in the limit remains a very interesting one, we mention the paper of Parreau \cite{PAR12} which gives an algebraic/topological interpretation in terms of actions on Buildings.  

\subsection{Hitchin representations as metrics on surfaces}

Our final theorem gives a way to interpret Hitchin representations as Riemannian metrics on surfaces.  This theorem is not very satisfying, as we would like to dispense with the dependence on the complex structure, but as it stands it is the best we can do.  In what follows $\mathcal{M}(\Sigma)$ denotes the set of Riemannian metrics on $\Sigma$ up to isotopy.  
\begin{theorem}
Fix a complex adjoint simple Lie group $G,$ and let $\text{\textnormal{Hit}}(G)$ be the space of all Hitchin representation considered up to conjugacy.  Then for each choice of marked complex structure $\sigma\in \mathcal{T}$ (where $\mathcal{T}$ is the Teichm\"uller space of $\Sigma),$ there is a mapping class group equivariant mapping
\begin{align}
\mathfrak{U}_{\sigma}:\text{\textnormal{Hit}}(G)\rightarrow \mathcal{M}(\Sigma).
\end{align}
\end{theorem}
\textbf{Remark:}  It would take us too far afield to attempt a careful proof here, but the above map should be real analytic.  We give a brief outline: it is well-known that the non-abelian Hodge correspondence identifiying the conjugacy classes of irreducible representations with the set of stable Higgs bundles is real analytic.  The harmonic map is constructed by solving a semi-linear system of elliptic partial differential equations and hence by the regularity theory developed by Morrey, the harmonic map depends real analytically on the coefficients of that elliptic system.  As pulling back the metric from the symmetric space is a quadratic expression in the derivatives of the harmonic map, this pull-back also depends real analytically on the coefficients.  This completes the sketch.
\begin{proof}
By Theorem \ref{Main}, for any $\rho\in \text{Hit}(G)$ the corresponding equivariant harmonic map,
\begin{align}
f_{\rho}:(\widetilde{\Sigma},\sigma)\rightarrow G/K,
\end{align}
is an immersion.  The Killing form on $\mathfrak{g}$ induces a left-invariant Riemannian metric on $G/K$ which we denote by $H_G.$  Pulling back $H_{G}$ via the immersion $f_{\rho}$
defines a Riemannian metric on the universal cover $\widetilde{\Sigma}.$  Moreover, by equivariance the deck group $\pi_1(\Sigma)$ acts via isometries and this metric descends to the quotient surface $\Sigma;$ this recipe defines $\mathfrak{U}_{\sigma}(\rho).$  

Now, suppose that $\eta: \Sigma\rightarrow \Sigma$ is a diffeomorphism and $\widetilde{\eta}:\widetilde{\Sigma}\rightarrow \widetilde{\Sigma}$ a lift to the universal cover.  Then,
\begin{align}
f_{\rho}(\widetilde{\eta}(\gamma(x)))=f_{\rho}(\eta_{*}(\gamma)(x))=\rho(\eta_{*}(\gamma))f_{\rho}(x),
\end{align}
for every $\gamma\in \pi_1(\Sigma)$ acting by deck transformations.  Thus, the equivariant map for the representation $\rho \circ \eta_{*}$ is $f_{\rho}\circ \widetilde{\eta}.$  Hence, the corresponding pullback metric changes by pullback via $\eta.$  This proves the mapping class group equivariance.
\end{proof}

\textbf{Remark:}  As mentioned in the introduction, unless $G=\text{PSL}_{2}(\mathbb{C}),$ the above map will not be injective.  The $\mathbb{S}^1$-action given by multiplying the Higgs field by $e^{i\theta}$ produces non-trivial (provided the rank is greater than one) one-parameter families of harmonic maps which all have the same energy.  If the harmonic map is in addition conformal, these conformal harmonic maps (hence minimal surfaces) all have isometric pullback metrics.  The families thus obtained are well known in the integrable systems literature where they are called the associated family of harmonic maps.  

\textbf{Remark:}  This mapping suggests the possibility of studying the geometry of the space of Hitchin representations via the study of this space of Riemannian metrics.  That being said, developing workable expressions for the Riemannian metrics one obtains in this fashion seems a challenging proposal.

\bibliography{DDBib}{}
\bibliographystyle{alpha}

\end{document}